\documentclass[]{interact}

\usepackage{epstopdf}% To incorporate .eps illustrations using PDFLaTeX, etc.
\usepackage[caption=false]{subfig}% Support for small, `sub' figures and tables

\usepackage[numbers,sort&compress]{natbib}% Citation support using natbib.sty
\bibpunct[, ]{[}{]}{,}{n}{,}{,}% Citation support using natbib.sty
\makeatletter% @ becomes a letter
\def\NAT@def@citea{\def\@citea{\NAT@separator}}% Suppress spaces between citations using natbib.sty
\makeatother% @ becomes a symbol again
\usepackage{hyperref}
\hypersetup{
	colorlinks=true,
	linkcolor=blue, % Couleur des liens internes
	citecolor=red, % Couleur des num?ros de la biblio dans le corps
	urlcolor=blue  } % Couleur des url
\usepackage{mathptmx}   %Times New Roman

\theoremstyle{plain}% Theorem-like structures provided by amsthm.sty
\newtheorem{theorem}{Theorem}[section]
\newtheorem{lemma}[theorem]{Lemma}
\newtheorem{corollary}[theorem]{Corollary}

\theoremstyle{definition}
\newtheorem{definition}[theorem]{Definition}
\newtheorem{example}[theorem]{Example}

\theoremstyle{remark}
\newtheorem{remark}{Remark}

\renewcommand {\theenumi} {\rm\roman{enumi}}

\newcommand{\al}{\alpha}
\newcommand{\be}{\beta}
\newcommand{\ga}{\gamma}

\newcommand{\de}{\delta}
\newcommand{\eps}{\varepsilon}
\newcommand{\bx}{\bar x}
\newcommand{\by}{\bar y}

\newcommand {\R} {\mathbb R}
\newcommand {\N} {\mathbb N}

\newcommand {\B} {\mathbb B}

\newcommand {\gph} {{\rm gph}\,}%Graph

\newcommand {\cl} {{\rm cl}\,}

\newcommand {\Int} {{\rm int}\,}

\newcommand{\toto}{\rightrightarrows}% geht nur mit amssymb.sty
\newcommand{\folgt}{$ \Rightarrow\ $}

\def\es{\emptyset}

\def\RHS{right-hand side}
\def\SVM{set-valued mapping}
\def\EVP{Ekeland variational principle}
\def\Fr{Fr\'echet}
\newcommand{\norm}[1]{\left\Vert#1\right\Vert}

\newcommand{\ang}[1]{\left\langle #1 \right\rangle}
\newcommand{\qdtx}[1]{\quad\mbox{#1}\quad}
\newcommand{\AND}{\quad\mbox{and}\quad}
\newcounter{mycount}


\begin{document}
	
%\title{Extremality of collections of sets with respect to general perturbations}
\title{Extremality of families of sets}

\author{
	\name{Nguyen Duy Cuong\textsuperscript{a}, Alexander Y. Kruger\textsuperscript{b}, and Nguyen Hieu Thao\textsuperscript{c}}
	\thanks{Corresponding author: Alexander Y. Kruger. Email: alexanderkruger@tdtu.edu.vn}
	\thanks{Dedicated to the memory of Prof. Diethard Pallaschke, a great scholar and good friend}
	\affil{\textsuperscript{a} Department of Mathematics, College of Natural Sciences, Can Tho University, Can Tho, Vietnam, ORCID: 0000-0003-2579-3601\\
		\textsuperscript{b} Optimization Research Group,
		Faculty of Mathematics and Statistics,
		Ton Duc Thang University,
		Ho Chi Minh City, Vietnam, ORCID: 0000-0002-7861-7380\\
\textsuperscript{c} School of Science, Engineering and Technology, RMIT University Vietnam, Ho Chi Minh City, Vietnam, ORCID: 0000-0002-1455-9169}
}
\maketitle

\begin{abstract}
The paper proposes another extension of the
\emph{extremal principle}.
A new extremality model involving collections of arbitrary families of sets is studied.
It generalizes the conventional model based on linear translations of given sets as well as its set-valued extensions.
This approach leads to a more general and simpler version of \emph{fuzzy separation}.
The new model is capable of treating a wider range of optimization and variational problems.
%We demonstrate its applicability to set-valued optimization problems with general preferences, weakening the assumptions of the known results and streamlining their proofs.
\end{abstract}


\begin{amscode}
	49J52; 49J53; 49K40; 90C30; 90C46
\end{amscode}

%\setcounter{tocdepth}{2}
%\tableofcontents

\section{Introduction and preliminaries}

The paper proposes another extension of the
\emph{extremal principle}, which was introduced more than 40 years ago in Kruger \& Mordukhovich \cite{KruMor80} (under the name \emph{generalized Euler equation}) as a variational counterpart of the
convex separation theorem in nonconvex settings and
has proved to be one of the fundamental results of variational analysis.
It serves as a powerful tool for proving necessary optimality conditions, subdifferential, normal cone and coderivative calculus formulas as well as many other results.
We refer the readers to the books \cite{Mor06.1,Mor06.2} and the bibliography therein for a comprehensive exposition of the history, motivations and various applications of the extremal principle.

The key point of the extremal principle is the geometric concept of \emph{extremal collection of sets}, which embraces various notions of optimality in extremal problems and is applicable in many other situations.
In the next definition and throughout the paper we consider a collection of $n>1$ arbitrary nonempty subsets $\Omega_1,\ldots,\Omega_n$ of a normed space and write $\{\Omega_1,\ldots,\Omega_n\}$ to denote the collection of sets as a single object.

\begin{definition}
[Extremality]
\label{D1.1}
The collection $\{\Omega_1,\ldots,\Omega_n\}$
is extremal at $\bx\in\cap_{i=1}^n\Omega_i$ if there exists a $\rho\in(0,+\infty]$ such that,
for any $\varepsilon>0$, there exist vectors $a_i$ $(i=1,\ldots,n)$ such that
$\max_{1\le i\le n}\|a_i\|<\varepsilon$ and
$\bigcap_{i=1}^n(\Omega_i-a_i)\cap {B_\rho(\bx)=\emptyset}$.
\end{definition}

The above definition (as well as some of its extensions below) covers both local ($\rho<+\infty$) and global ($\rho=+\infty$) extremality.
In the latter case, the point $\bx$ plays no role apart from ensuring that $\cap_{i=1}^n\Omega_i\ne\es$.

The dual counterpart of the extremality property in the next assertion (it is not a theorem!) constitutes a kind of fuzzy separation.
%of a collection of (not necessarily convex) sets.
It employs \emph{\Fr\ normal cones}.
\medskip

\noindent
\textbf{Extremal principle}
\emph{Let $\Omega_1,\ldots,\Omega_n$
be closed, and $\bx\in\cap_{i=1}^n\Omega_i$.
If $\{\Omega_1,\ldots,\Omega_n\}$ is extremal at $\bx$, then, for any $\varepsilon>0$, there exist $x_i\in \Omega_i\cap B_{\varepsilon}(\bx)$ and $x_i^*\in N^F_{\Omega_i}(x_i)$ $(i=1,\ldots,n)$ such that
$\|\sum_{i=1}^nx_i^*\|<\varepsilon$ and $\sum_{i=1}^n\|x_i^*\|=1$.}
\medskip
%\NDC{11/3/24.
%in Asplund space?}

The extremal principle was first established in  \cite{KruMor80} under the assumption that the space admits an equivalent norm,
%that is
\Fr\ differentiable away from zero, and the fuzzy separation was formulated using, instead of \Fr\ normal cones, their certain enlargements called \emph{sets of \Fr\ $\eps$-normals}.
Employing the \emph{fuzzy sum rule} for \Fr\ subdifferentials established in Fabian \cite{Fab89}, the extremal principle was extended in Mordukhovich \& Shao \cite{MorSha96} (where to the best of our knowledge the name \emph{extremal principle} first appeared) to Asplund spaces, and the $\eps$-normals in its statement were replaced by the conventional \Fr\ normals.
It was also shown in \cite{MorSha96} that the extremal principle in terms of \Fr\ normals cannot be extended beyond Asplund spaces.

\begin{theorem}
[Extremal characterization of Asplund spaces]
\label{T1.1}
The extremal principle holds true for any collection of closed subsets of a Banach space if and only if the space is Asplund.
\end{theorem}

Recall that a Banach space is \emph{Asplund} if every continuous convex function on an open convex set is Fr\'echet differentiable on a dense subset \cite{Phe93}, or equivalently, if the dual of each its separable subspace is separable.
We refer the reader to \cite{Phe93,Mor06.1} for discussions
about
and
characterizations of Asplund spaces.
All reflexive, particularly, all finite dimensional Banach spaces are Asplund.
\begin{remark}
\label{R1.1}
It is not uncommon to employ in the statements of the extremal principle and its extensions an alternative form of the fuzzy \Fr\ separation property:
\smallskip

\leftskip=1cm
\rightskip=1cm
\noindent
\emph{for any $\eps>0$, there exist
$x_i\in\Omega_i\cap{B}_\eps(\bx)$ and vectors $x_i^*\in X^*$ ${(i=1,\ldots,n)}$ such that
$\sum_{i=1}^nd(x^*_i,N^F_{\Omega_i}(x_i))<\eps$,
$\sum_{i=1}^n x^*_i=0$ and
$\sum_{i=1}^{n}\norm{x_i^*}=1$.}
\smallskip

\leftskip=0cm
\rightskip=0cm

\noindent
The equivalence of the two forms of fuzzy separation is easy to check; see
\cite[Theorem~1]{BuiKru18}.
The fuzzy separation conditions similar to those in the Extremal principle and the one above have been used interchangeably in generalized separation statements for decades; cf. \cite{KruMor80,KruMor80.2,Kru81.2,Kru85.1, MorSha96,BorZhu05,Mor06.1,Mor06.2}.
\end{remark}

The fuzzy \Fr\ separation in the extremal principle is a necessary condition for extremality.
It is naturally valid under weaker assumptions.
The extremality property was relaxed to
%\emph{local extremality} \cite{Kru81.2} and then
\emph{stationarity} and \emph{approximate stationarity} \cite{Kru98,Kru02,Kru03,Kru04,Kru05,Kru06,Kru09}, while preserving the fuzzy \Fr\ separation conclusion, and without significant changes in the original proof.

\begin{definition}
[Approximate stationarity]
\label{D1.3}
The collection $\{\Omega_1,\ldots,\Omega_ n\}$ is
approximately stationary at $\bx\in\cap_{i=1}^n\Omega_i$ if for any $\varepsilon>0$, there exist $\rho\in(0,\varepsilon)$, $x_i \in \Omega_i \cap B_\eps(\bx)$, and vectors $a_i\in X$ $(i=1,\ldots,n)$ such that
$\max_{1\le i\le n}\|a_i\|<\varepsilon\rho$ and
$\bigcap_{i=1}^n(\Omega_i-x_i-a_i)\cap(\rho \mathbb{B})=\emptyset$.
\end{definition}

In the particular case when $x_1=\ldots=x_n=\bx$ in Definition~\ref{D1.3}, the last non-intersection condition takes a simpler form: $\bigcap_{i=1}^n(\Omega_i-a_i)\cap B_\rho(\bx)=\emptyset$.
The corresponding property is referred to in \cite{Kru98,Kru02,Kru03,Kru04,Kru05,Kru06,Kru09} as simply \emph{stationarity}.
The latter stronger property is obviously implied by the extremality in Definition~\ref{D1.1}.
When the sets are convex, all the properties in Definitions~\ref{D1.1} and \ref{D1.3} are equivalent; cf. \cite[Proposition~14]{Kru05}.

The approximate stationarity is sufficient for the fuzzy \Fr\ separation in the conclusion of the extremal principle.
In Asplund spaces, the two properties are actually equivalent.
This result is known as the {extended extremal principle} \cite{Kru98,Kru02,Kru03}.

\begin{theorem}
[Extended extremal principle]
\label{T1.2}
Let $\Omega_1,\ldots,\Omega_n$
be closed subsets of an Asplund space, and $\bx\in\cap_{i=1}^n\Omega_i$.
The collection $\{\Omega_1,\ldots,\Omega_n\}$ is approximately stationary at $\bx$ if and only if, for any $\varepsilon>0$, there exist $x_i\in \Omega_i\cap B_{\varepsilon}(\bx)$ and $x_i^*\in N^F_{\Omega_i}(x_i)$ $(i=1,\ldots,n)$ such that
$\|\sum_{i=1}^nx_i^*\|<\varepsilon$ and $\sum_{i=1}^n\|x_i^*\|=1$.
\end{theorem}

The necessity of the fuzzy separation for the approximate stationarity (hence, also for the extremality) can be easily extended to general Banach spaces if \Fr\ normal cones are replaced by other normal cones corresponding to subdifferentials possessing a satisfactory calculus in Banach spaces, e.g., Clarke or Ioffe normal cones; cf. \cite[Remark~2.1(iii)]{BuiKru19}.
The converse implication in Theorem~\ref{T1.2} is only true for \Fr\ normal cones.

The core feature of the definitions of extremality and stationarity is the non-inter\-section of certain perturbations
%(deformations)
of either the original sets $\Omega_1,\ldots,\Omega_n$ in Definition~\ref{D1.1} or their translations $\Omega_1-x_1,\ldots,\Omega_n-x_n$ in Definition~\ref{D1.3}.
The perturbations employed in the definitions are in the form of linear translations of the sets: given a set $\Omega$, its linear translation is the set $\Omega-a$ for some vector $a$.

The conventional extremal principle as well as its extension formulated above, with perturbations in the form of linear translations, have proved to be versatile enough to cover a wide range of problems in optimization and variational analysis as demonstrated, e.g., in the books \cite{BorZhu05,Mor06.1,Mor06.2}.
At the same time, there exist problems, mainly in multiobjective and set-valued optimization described via closed
preference relations, that cannot be covered within the framework of linear translations.
The first example of this kind was identified in Zhu \cite{Zhu00}.
Fortunately, such problems can be handled with the help of a more flexible extended version of the extremal principle using more general nonlinear perturbations (deformations) of the sets defined by set-valued mappings.
Such an extension was developed in Mordukhovich et al. \cite{MorTreZhu03} (see also \cite{BorZhu05,Mor06.2}) and applied by other authors to various multiobjective optimization problems in \cite{ZheNg06,LiNgZhe07,Bao14.2}.
Below is our interpretation of the corresponding definitions from \cite{MorTreZhu03,Mor06.2} complying with the notation and terminology of Definitions~\ref{D1.1} and \ref{D1.3}.
%It talks about set-valued (SV) extremality.

\begin{definition}
[Extremality: set-valued perturbations]
\label{D1.4}
Let $\Omega_1,\ldots,\Omega_n$ be subsets of a normed space $X$, $\bx\in\cap_{i=1}^n\Omega_i$, and,
for each $i=1,\ldots,n$,
$S_i:M_i\toto X$ is a \SVM\ from a metric space
$(M_i,d_i)$ to $X$ and $S_i(\bar s_i)=\Omega_i$ for some $\bar s_i\in M_i$.
The collection $\{\Omega_1,\ldots,\Omega_n\}$
is extremal at $\bx$ with respect to $\{S_1,\ldots,S_n\}$ if there exists a $\rho\in(0,+\infty]$ such that, for any $\varepsilon>0$, there exist $s_i\in M_i$ $(i=1,\ldots,n)$ such that
$\max_{1\le i\le n}d_i(s_i,\bar s_i)<\varepsilon$,
$\max_{1\le i\le n}d(\bx,S_i(s_i))<\varepsilon$
and
$\bigcap_{i=1}^nS_i(s_i)\cap B_\rho(\bx){=\emptyset}$.
\end{definition}

%\begin{remark}
%\label{R1.1}
%The property in Definition~\ref{D1.4} can be formulated with $X$ assumed an arbitrary metric space.
%\end{remark}

Similar to Definitions~\ref{D1.1} and \ref{D1.3}, the model in Definition~\ref{D1.4} exploits non-inter\-section of perturbations (deformations) of the given sets $\Omega_1,\ldots,\Omega_n$.
The perturbations are chosen from the respective families of sets $\{S_i(s)\mid s\in M_i\}$
$(i=1,\ldots,n)$
determined by the given \SVM s $S_1,\ldots,S_n$.
In the particular case of linear translations, i.e., when, for all $i=1,\ldots,n$, $(M_i,d_i)=(X,d)$
%\NDC{11/3/24.
%$(X,\|\cdot\|)$?
%I think the property also makes sense when $X$ is a metric space.}
and $S_i(a)=\Omega_i-a_i$ $(a_i\in X)$, we have $S_i(0)=\Omega_i$, and the property reduces to the one in Definition~\ref{D1.1}.
It was shown by examples in \cite{MorTreZhu03,Mor06.2} that in general the framework of Definition~\ref{D1.4} is richer than that of Definition~\ref{D1.1}.
With minor modifications in the proof, the conventional extremal principle can be extended to the setting of Definition~\ref{D1.4} producing a more advanced version of fuzzy separation; cf. \cite{MorTreZhu03,Mor06.2}.

\begin{theorem}
[Extremal principle: set-valued perturbations]
\label{T1.3}
Let $\Omega_1,\ldots,\Omega_n$ %$(n>1)$
be subsets of an Asplund space $X$, $\bx\in\cap_{i=1}^n\Omega_i$, and,
for each $i=1,\ldots,n$,
$S_i:M_i\toto X$ be closed-valued, and $S_i(\bar s_i)=\Omega_i$ for some $\bar s_i\in M_i$, where
$(M_i,d_i)$ is a metric space.
Suppose that $\{\Omega_1,\ldots,\Omega_n\}$ is extremal at $\bx$ with respect to $\{S_1,\ldots,S_n\}$.
Then, for any $\varepsilon>0$, there exist $s_i\in M_i\cap B_{\eps}(\bar s_i)$, $x_i\in S_i(s_i)\cap B_\varepsilon(\bx)$, and $x_i^*\in N^F_{S_i(s_i)}(x_i)$ $(i=1,\ldots,n)$ such that
$\|\sum_{i=1}^nx_i^*\|<\varepsilon$ and $\sum_{i=1}^n\|x_i^*\|{=1}$.
\end{theorem}

In the particular case of linear translations, the conclusion of Theorem~\ref{T1.3} reduces to the fuzzy \Fr\ separation in the Extremal principle.
The fuzzy separation was formulated in \cite{MorTreZhu03,Mor06.2} in a slightly different but equivalent form as commented on in Remark~\ref{R1.1}.
The statement of Theorem~\ref{T1.3} in its current form cannot be extended beyond Asplund spaces, and an analogue of Theorem~\ref{T1.1} holds true for the extremal principle in Theorem~\ref{T1.3}; cf. \cite[Theorem~5.68]{Mor06.2}.

%The extremality and fuzzy separation
%model based on set-valued perturbations in Definition~\ref{D1.4} and Theorem~\ref{T1.3}
%has proved to be useful when dealing with multiobjective and set-valued optimization problems with general preference relations, multiplayer games, and optimal control problems; see \cite{Zhu00,MorTreZhu03,Mor06.2, ZheNg06,LiNgZhe07,Bao14.2}.
In this paper, we refine the model in Definition~\ref{D1.4} and Theorem~\ref{T1.3} making it more flexible and, at the same time, simpler.
%We allow perturbations (deformations) of the given sets to be chosen from the given nonempty families of arbitrary sets.
%As a consequence, there is no need to employ \SVM s $S_1,\ldots,S_n$ and auxiliary spaces $M_1,\ldots,M_n$.
%\todo{We give no evidence of that}
Instead of the \SVM s $S_1,\ldots,S_n$, we consider collections of families of arbitrary sets and prove a corresponding extremal principle.
The resulting more general model is applicable to a wider range of variational problems, weakens their assumptions and streamlines the proofs.

As mentioned above, the core feature of Definitions~\ref{D1.1}, \ref{D1.3} and \ref{D1.4} as well as all existing definitions of extremality and stationarity in the literature is the non-inter\-section of certain perturbations of the given sets.
Accordingly, all existing proofs of the extremal principle and its extensions first establish necessary conditions for non-inter\-section of arbitrary collections of sets.
It was first observed by Zheng \& Ng in \cite{ZheNg05.2} that such necessary conditions can be of interest on their own as powerful tools for proving various `extremal' results.
They formulated two abstract non-inter\-section lemmas and demonstrated their applications in optimization.
This fruitful idea has been further developed in \cite{ZheNg06,LiNgZhe07,LiTanYuWei08} culminating in a \emph{unified separation theorem} for closed sets in a Banach space in \cite{ZheNg11}.
The approach has been analyzed and refined in \cite{BuiKru18,BuiKru19}.
In this paper, we formulate a slightly more refined generalized separation result.

%\section{Preliminaries}\label{S2}

Our basic notation is standard, see, e.g., \cite{Mor06.1,RocWet98,DonRoc14,Iof17}.
Throughout the paper, if not explicitly stated otherwise, $X$ and $Y$ are  normed spaces.
Products of normed spaces are assumed to be equipped with the maximum norm.
The topological dual of a normed space $X$ is denoted by $X^*$, while $\langle\cdot,\cdot\rangle$ denotes the bilinear form defining the pairing between the two spaces.
The open ball with center $x$ and radius $\delta>0$ is denoted by $B_\delta(x)$.
%If $(x,y)\in X\times Y$, we write $B_\varepsilon(x,y)$ instead of $B_\varepsilon((x,y))$.
The open unit ball is denoted by $\B$ with a subscript indicating the space, e.g., $\B_X$ and $\B_{X^*}$.
Symbols $\R$ and $\N$ stand for the sets of all, respectively, real and positive integer numbers.

The interior and closure of a set $\Omega$ are denoted by $\Int\Omega$ and $\cl\Omega$, respectively.
The distance from a point $x \in X$ to a subset $\Omega\subset X$ is defined by $d(x,\Omega):=\inf_{u \in \Omega}\|u-x\|$, and we use the convention $d(x,\emptyset)=+\infty$.
Given a subset $\Omega$ of a normed space $X$ and a point $\bx\in \Omega$, the sets
\begin{gather}\label{NC}
N_{\Omega}^F(\bx):= \Big\{x^\ast\in X^\ast\mid
\limsup_{\Omega\ni x{\rightarrow}\bar x,\;x\ne\bx} \frac {\langle x^\ast,x-\bx\rangle}
{\|x-\bx\|} \le 0 \Big\},
\\\label{NCC}
N_{\Omega}^C(\bx):= \left\{x^\ast\in X^\ast\mid
\ang{x^\ast,z}\le0
\qdtx{for all}
z\in T_{\Omega}^C(\bx)\right\}
\end{gather}
are the \emph{Fr\'echet} and \emph{Clarke normal cones} (cf. \cite{Kru03,Cla83}) to $\Omega$ at $\bx$,
%The notation $x\stackrel{\Omega}{\rightarrow}\bar x$ in \eqref{NC} means $x \rightarrow \bar x$ and $x \in \Omega \setminus\{\bx\}$, and
where $T_{\Omega}^C(\bx)$
%in \eqref{NCC}
stands for the \emph{Clarke tangent cone} to $\Omega$ at $\bx$:
\begin{multline*}%\label{TCC}
T_{\Omega}^C(\bx):= \big\{z\in X\mid
\forall x_k{\rightarrow}\bx,\;x_k\in\Omega,\;\forall t_k\downarrow0,\;\exists z_k\to z\\
\mbox{such that}\quad
x_k+t_kz_k\in \Omega \qdtx{for all}
k\in\N\big\}.
\end{multline*}
The sets \eqref{NC} and \eqref{NCC} are nonempty
closed convex cones satisfying $N_{\Omega}^F(\bx)\subset N_{\Omega}^C(\bx)$.
If $\Omega$ is a convex set, they reduce to the normal cone in the sense of convex analysis:
$N_{\Omega}(\bx):= \left\{x^*\in X^*\mid \langle x^*,x-\bx \rangle \leq 0 \qdtx{for all} x\in \Omega\right\}$.

Structure of the paper.
%Section~\ref{S2} recalls some definitions and facts used throughout the paper.
In Section~\ref{S3}, we formulate a generalized separation result for a certain $(\eps,\rho)$-extre\-mality property of a collection of sets.
%The other dual conditions in the rest of the paper are consequences of this result.
%In particular, i
It is used in Section~\ref{S4} to prove an extremal principle for collections of families of arbitrary sets.
%The extended extremal principle in Theorem~\ref{T4.5} is an immediate consequence of the latter general result.
%We also define and provide sufficient conditions for the \emph{$\al$-transversality} property of collections of sets.

\section{Generalized separation}\label{S3}

The next abstract generalized separation statement is a refined version of \cite[Theorem 6.3]{BuiKru19}, which in turn refines several statements of this kind in the literature, in particular, the \emph{unified separation theorems} due to Zheng and Ng \cite[Theorems~3.1 and 3.4]{ZheNg11}.\footnote{
The mentioned results are essentially equivalent (observed by the reviewer).}
The result is a consequence of the \EVP\ and the corresponding subdifferential sum rules, and covers the conventional extremal principle as well as its existing extensions.
The proof follows the standard procedure, and we omit it for brevity.
The other dual conditions in the rest of the paper are consequences of this lemma.

\begin{lemma}
[Generalized separation]
\label{T5.1}
Let $\Omega_1,\ldots,\Omega_n$ %$(n>1)$
be closed subsets of a Banach space $X$,
$\omega_i\in\Omega_i$ $(i=1,\ldots,n)$, and
$\eps>0$.
Suppose that
$\bigcap_{i=1}^n\Omega_i=\emptyset$ and
\begin{gather}
\label{T3.1-1}
\max_{1\le i\le n-1} \|\omega_i-\omega_n\| <\inf_{u_i\in\Omega_i\;(i=1,\ldots,n)}\max_{1\le i\le n-1}\|u_i-u_n\| +\varepsilon.
\end{gather}
Then, for any $\de>0$ and $\eta>0$, there exist $x_i\in\Omega_i$, $x_i^*\in X^*$ $(i=1,\ldots,n)$ such that
\begin{gather}
\label{T3.1-2}
\max_{1\le i\le n-1}\|x_i-\omega_i\|<\de,\quad
\|x_n-\omega_n\|<\eta,\quad
\sum_{i=1}^n x_i^*=0,\quad
\sum_{i=1}^{n-1}\|x_i^*\|=1,
\\
%\label{T5.1-1}
\notag
\de\sum_{i=1}^{n-1} d\left(x_i^*,N^C_{\Omega_i}(x_i)\right)+ \eta d\left(x_n^*,N^C_{\Omega_n}(x_n)\right) <\eps,
\\
\label{T5.1-2}
\sum_{i=1}^{n-1}\langle x_i^*,x_n-x_i\rangle=\max_{1\le i\le n-1} \|x_n-x_i\|.
\end{gather}

If $X$ is Asplund, then, for any $\de>0$, $\eta>0$ and $\tau\in(0,1)$, there exist $x_i\in\Omega_i$, $x_i^*\in X^*$ $(i=1,\ldots,n)$ satisfying \eqref{T3.1-2} and such that
\begin{gather}
%\label{T5.4-1}
\notag
\de\sum_{i=1}^{n-1} d\left(x_i^*,N^F_{\Omega_i}(x_i)\right)+ \eta d\left(x_n^*,N^F_{\Omega_n}(x_n)\right) <\eps,
\\
\label{T5.4-3}
\sum_{i=1}^{n-1}\langle x_i^*, x_n-x_i\rangle>\tau\max_{1\le i\le n-1} \|x_n-x_i\|.
\end{gather}
\end{lemma}

Lemma \ref{T5.1} characterizes the global non-intersection property $\bigcap_{i=1}^n\Omega_i=\emptyset$.
%For many applications (like optimality conditions) weaker local properties are important.
Interestingly, necessary conditions for local non-intersection properties follow from the global ones.
For what follows, we need a localized version of Lemma~\ref{T5.1} where each set $\Omega_i$ is considered near a given point $\bx_i$ $(i=1,\ldots,n)$.
It follows from Lemma~\ref{T5.1} applied to the collection of $n+1$ closed set $\Omega_i':=\Omega_i-\bx_i$ $(i=1,\ldots,n)$ and $\Omega_{n+1}':=\eta\overline\B$, where $\eta\in(0,\rho)$ is arbitrarily close to $\rho$, and the points $\omega_i':=\omega_i-\bx_i\in\Omega_i'$ $(i=1,\ldots,n)$ and
$\omega_{n+1}':=0$.

\begin{corollary}
\label{C3.20}
Let $\Omega_1,\ldots,\Omega_n$ %$(n>1)$
be closed subsets of a Banach space $X$, $\omega_i\in\Omega_i$, $\bx_i\in X$ $(i=1,\ldots,n)$,
$\eps>0$ and $\rho>0$.
Suppose that
${\bigcap}_{i=1}^n(\Omega_i-\bx_i)\cap (\rho\B)=\emptyset$ and
\begin{gather}
\label{C3.2-10}
\max_{1\le i\le n}\|\bx_i-\omega_i\| <\inf_{u\in\rho\B}\;\max_{1\le i\le n} d(\bx_i +u,\Omega_i)+\varepsilon.
\end{gather}
Then, for any $\de>0$, there exist
$x_i\in\Omega_i$,
$x_i^*\in X^*$ $(i=1,\ldots,n)$ and {$x\in X$} such that
\begin{gather}
\label{C3.1-2}
\max_{1\le i\le n}\|x_i-\omega_i\|<\de,\quad
{\|x\|<\rho},\quad
\sum_{i=1}^{n}\|x_i^*\|=1,
\\
\label{C3.1-3}
\de\sum_{i=1}^n d\left(x_i^*,N^C_{\Omega_i}(x_i)\right)+ \rho\Big\|\sum_{i=1}^nx_i^*\Big\| <\eps,
\\
\label{C3.2-2}
\sum_{i=1}^{n}\langle x_i^*,x+\bx_i-x_i\rangle=\max_{1\le i\le n} \|x+\bx_i-x_i\|.
\end{gather}
If $X$ is Asplund, then, for any $\de>0$ and $\tau\in(0,1)$, there exist $x_i\in\Omega_i$,
$x_i^*\in X^*$ $(i=1,\ldots,n)$ and {$x\in X$} such that condition \eqref{C3.1-2} is satisfied, and
\begin{gather}
%\label{C3.1-5}
\notag
\de\sum_{i=1}^n d\left(x_i^*,N^F_{\Omega_i}(x_i)\right)+ \rho\Big\|\sum_{i=1}^nx_i^*\Big\| <\eps,
\\
\label{C3.2-3}
\sum_{i=1}^{n}\langle x_i^*,x+\bx_i-x_i\rangle >\tau\max_{1\le i\le n}\|x+\bx_i-x_i\|.
\end{gather}
\end{corollary}

\begin{remark}
\label{R3.1}
\begin{enumerate}
\item
The expression
$
\inf\limits_{u_i\in\Omega_i\;(i=1,\ldots,n)} \max\limits_{1\le i\le n-1}\|u_i-u_n\|
$
in the \RHS\ of \eqref{T3.1-1} corresponds to taking $p=+\infty$
%\NDC{11/3/24.I think it should be $p=1$. Also, they employ the sum norm in the product primal space.}
in the definition of the \emph{$p$-weighted nonintersect index} \cite[p.~890]{ZheNg11} of $\Omega_1,\ldots,\Omega_n$:
\sloppy
\begin{gather*}
\ga_p(\Omega_1,\ldots,\Omega_n):=
\inf_{u_i\in\Omega_i\;(i=1,\ldots,n)} \left(\sum_{i=1}^{n-1}\|u_i-u_n\|^p\right)^{1/p}.
\end{gather*}
It is easy to see that $\ga_p(\Omega_1,\ldots,\Omega_n)=0$ when $\bigcap_{i=1}^n\Omega_i\ne\emptyset$.
We restrict the presentation to the case $p=+\infty$ for simplicity.

\item
\label{R3.1.4}
When applying Lemma~\ref{T5.1} and Corollary~\ref{C3.20}, it can be convenient to replace \eqref{T3.1-1} and \eqref{C3.2-10} by  simpler (and slightly stronger) conditions $\max_{1\le i\le n-1} \|\omega_i-\omega_n\| <\varepsilon$ and $\max_{1\le i\le n} \|\bx_i-\omega_i\|<\varepsilon$, respectively.

\item
Conditions \eqref{T5.1-2} and \eqref{T5.4-3} relate the dual space vectors $x_i^*$ and the primal space vectors $x_n-x_{i}$ $(i=1,\ldots,n-1)$.
Such conditions, though not common in the conventional formulations of the extremal principle/generalized separation statements, seem to provide important additional characterizations of the properties.
Conditions of this kind first appeared explicitly in the generalized separation theorems in \cite{ZheNg11}, where the authors also provided
motivations for
employing such conditions.
\item
As $\tau$ in the second parts of Lemma~\ref{T5.1} and Corollary~\ref{C3.20} can be taken arbitrarily close to $1$, conditions \eqref{T5.4-3} and \eqref{C3.2-3} are approximate versions of conditions \eqref{T5.1-2} and \eqref{C3.2-2}, respectively.

\item
\label{R3.1.6}
Conditions \eqref{T5.1-2}, \eqref{T5.4-3}, \eqref{C3.2-2} and \eqref{C3.2-3} can be dropped from the statements of Lemma~\ref{T5.1} and Corollary~\ref{C3.20} (together with the number $\tau$ in the Asplund space parts) leading to more traditional (though weaker) necessary conditions.
This observation is applicable to all the consequences of these statements in the remainder of the paper.
\end{enumerate}
\end{remark}

\section{Extremal principle for a collection of families of sets}\label{S4}

In this section, we revisit and refine the extremality model with set-valued perturbations in Definition~\ref{D1.4} and Theorem~\ref{T1.3} making it simpler and more flexible.

Definition~\ref{D1.4} talks about extremality of a collection of sets but in fact it is about certain properties of a collection of \SVM s
$S_i:M_i\toto X$ $(i=1,\ldots,n)$, loosely connected with the given sets.
For each $i=1,\ldots,n$ and each $s_i\in M_i$, the set $S_i(s_i)$ is considered as a perturbation (deformation) of the set $S_i(\bar s_i)$, where $\bar s_i$ is a given point in $M_i$.
However, with no continuity assumptions on $S_i$ in the definition, the sets $S_i(s_i)$ and $S_i(\bar s_i)$ can be very different as illustrated by examples in \cite{MorTreZhu03,Mor06.2}.
The only essential requirement on perturbations imposed by the definition is that the distance from $S_i(s_i)$ to a single given point in $\cap_{i=1}^n S_i(\bar s_i)$ must be small.
Another restriction on $\max_{1\le i\le n}d_i(s_i,\bar s_i)$ does not seem to be of importance in Definition~\ref{D1.4} and Theorem~\ref{T1.3}.
Below we provide another simple example in $\R$.

\begin{example}
\label{E4.1}
Let $\Omega_1=\Omega_2:=\R$.
The pair $\{\Omega_1,\Omega_2\}$ is clearly neither
extremal at any point in $\R$ in the conventional sense of Definition~\ref{D1.1}
nor even approximately stationary in the sense of Definition~\ref{D1.3}.
Now set $M_1=M_2:=\R$, $S_1(0)=S_2(0):=\R$, $S_1(x):=\{0\}$ if $x\ne0$, $S_2(x):=\{x\}$ if $x\in\{1/n\mid n\in\N\}$, and $S_2(x):=\es$ elsewhere on $\R$.
Thus, $S_1(0)=\Omega_1$, $S_2(0)=\Omega_2$, while the other values of $S_1$ and $S_2$ have little to do with $\Omega_1$ and $\Omega_2$, respectively.
At the same time, $0\in S_1(x)$, $d(0,S_2(1/n))=1/n$ and $S_1(x)\cap S_2(1/n)=\es$ for all $x\ne0$ and $n\in\N$.
Since $1/n\to0$ as $n\to+\infty$, the pair $\{\Omega_1,\Omega_2\}$ is extremal at $0$ with respect to $\{S_1,S_2\}$ in the sense of Definition~\ref{D1.4} (even with $\rho=+\infty$).
\end{example}

One can conclude from the above observations that within the theory behind Definition~\ref{D1.4} and Theorem~\ref{T1.3} it suffices to examine ``extremality'' of the families of sets $\Xi_i:=\{S_i(s)\mid s\in M_i\}$ $(i=1,\ldots,n)$.
We are going to hone this theory further and drop the  \SVM s $S_i:M_i\toto X$ together with metric spaces $M_i$ $(i=1,\ldots,n)$ entirely, and study extremality of families $\Xi_i$ $(i=1,\ldots,n)$ of arbitrary subsets of $X$.
%
%We are going to allow perturbations (deformations) of the given sets $\Omega_1,\ldots,\Omega_n$ to be chosen from the given nonempty families of arbitrary sets $\Xi_i$ $(i=1,\ldots,n)$.
We make no assumptions on the number of members in $\Xi_i$: it can range from a single set to an infinite (possibly uncountable) number of sets.
This simplifies the theory and makes it applicable to a wider range of problems.

\begin{definition}
[Extremality and stationarity: families of sets]
\label{D4.1}
Let $\Xi_1,\ldots,\Xi_n$ be families of subsets of a normed space $X$, and $\bx\in X$.
\begin{enumerate}
\item
\label{D4.1.1}
The collection $\{\Xi_1,\ldots,\Xi_ n\}$ is
extremal at $\bx$ if there is a $\rho\in(0,+\infty]$ such that, for any $\varepsilon>0$, there exist $A_i\in\Xi_i$ $(i=1,\ldots,n)$ such that $\max_{1\le i\le n}d(\bx,A_i)< \varepsilon$ and
$\bigcap_{i=1}^nA_i\cap B_\rho(\bx)=\emptyset$.
\item
\label{D4.1.2}
Let $\Omega_1,\ldots,\Omega_n\subset X$.
The collection $\{\Xi_1,\ldots,\Xi_ n\}$ is
approximately stationary at $\bx$ with respect to $\{\Omega_1,\ldots,\Omega_n\}$ if, for any $\varepsilon>0$, there exist a $\rho\in(0,\varepsilon)$, $x_i\in\Omega_i\cap B_{\varepsilon}(\bx)$ and  $A_i\in\Xi_i$ $(i=1,\ldots,n)$ such that
$\max_{1\le i\le n}d(x_i, A_i)<\varepsilon\rho$ and
$\bigcap_{i=1}^n(A_i-x_i)\cap (\rho\B)=\emptyset$.
\item
\label{D4.1.3}
The collection $\{\Xi_1,\ldots,\Xi_ n\}$ is
stationary at $\bx$ if it is
approximately stationary at $\bx$ with respect to $\{\{\bx\},\ldots,\{\bx\}\}$.
\end{enumerate}
\end{definition}

\begin{remark}
\label{R3}
\begin{enumerate}
\item
\label{R3.01}
It is easy to see that \eqref{D4.1.1} \folgt \eqref{D4.1.3} in Definition~\ref{D4.1} and, if $\bx\in\cap_{i=1}^n\Omega_i$, then \eqref{D4.1.3} \folgt \eqref{D4.1.2}.
\item
\label{R3.02}
Definition~\ref{D4.1} covers the extremality and stationarity properties in Definitions~\ref{D1.1}, \ref{D1.3} and \ref{D1.4}.
The first two correspond to setting $\Xi_i:=\{\Omega_i-a\mid a\in X\}$ in parts \eqref{D4.1.1} and \eqref{D4.1.2}, respectively, while setting
$\Xi_i:=\{S_i(s)\mid s\in M_i\}$ in part \eqref{D4.1.1} covers the last one.
Note that, unlike Definitions~\ref{D1.1}, \ref{D1.3} and \ref{D1.4} which are about extremality and stationarity properties of a given collection of sets $\Omega_1,\ldots,\Omega_n$, Definition~\ref{D4.1} considers a more general setting of a given collection $\Xi_1,\ldots,\Xi_n$ of families of sets.

\item
Unlike the setting of Definition~\ref{D1.4}, we do not assume in Definition~\ref{D4.1} that $\Omega_i\in\Xi_i$ $(i=1,\ldots,n)$.
\end{enumerate}
\end{remark}

\begin{example}
Revisiting Example~\ref{E4.1}, we consider two families of subsets of $\R$.
Let $\Xi_1$ consist of a single one-point set $\{0\}$,  and $\Xi_2$ be a family of singletons $\{1/n\}$ for $n\in\N$.
It is easy to see that
$\{\Xi_1,\Xi_2\}$ is extremal at $0$ in the sense of Definition~\ref{D4.1}\,\eqref{D4.1.1} with $\rho=+\infty$.
\end{example}

\if{
The more general extremality and stationarity model in Definition~\ref{D4.1} leads to a more general and simpler version of fuzzy separation than that in Theorem~\ref{T1.3}.

\begin{definition}
[Fuzzy separation: general perturbations]
%\label{D1.6}
Let $\Xi_1,\ldots,\Xi_n$ be families of subsets of a normed space $X$, and $\bx\in X$.
The collection $\{\Xi_1,\ldots,\Xi_n\}$
satisfies the fuzzy \Fr\ separation property at $\bx$ if,
for any $\varepsilon>0$, there exist
$A_i\in\Xi_i$, $x_i\in A_i\cap B_\varepsilon(\bx)$ and $x_i^*\in N^F_{A_i}(x_i)$ $(i=1,\ldots,n)$ such that
$\|\sum_{i=1}^nx_i^*\|<\varepsilon$ and $\sum_{i=1}^n\|x_i^*\|=1$.

If the above property holds true with $N^C$ in place of $N^F$, we say that $\{\Xi_1,\ldots,\Xi_n\}$
satisfies the fuzzy Clarke separation at $\bx$.
\end{definition}
}\fi

The next theorem establishes a more general and simpler version of fuzzy separation than that in Theorem~\ref{T1.3}.

\begin{theorem}
\label{T4.5}
Let
$\Xi_1,\ldots,\Xi_n$ be families of closed subsets of a Banach space $X$, and $\bx\in X$.
Let $\Omega_1,\ldots,\Omega_n\subset X$.
If $\{\Xi_1,\ldots,\Xi_n\}$ is approximately stationary at $\bx$ with respect to $\{\Omega_1,\ldots,\Omega_n\}$,
then, for any $\varepsilon>0$ and $\tau\in(0,1)$, there exist $\bx_i\in\Omega_i\cap B_{\eps}(\bx)$, $A_i\in\Xi_i$,
$x_i\in A_i\cap B_\varepsilon(\bx)$, $x_i^*\in N^C_{A_i}(x_i)$ $(i=1,\ldots,n)$ and $x\in\varepsilon\B$, such that
$\|\sum_{i=1}^nx_i^*\|<\eps$, $\sum_{i=1}^n\|x_i^*\|=1$,
and condition \eqref{C3.2-3} is satisfied.

If $X$ is Asplund, then $N^C$ in the above assertion can be replaced by $N^F$.
\end{theorem}

\begin{proof}
Suppose $\{\Xi_1,\ldots,\Xi_n\}$ is approximately sta\-tionary at $\bx$ with respect to $\{\Omega_1,\ldots,\Omega_n\}$.
Let $\eps>0$ and $\tau\in(0,1)$.
Choose a positive number $\xi<\min\{\eps/4,(1-\tau)/2\}$.
By Definition~\ref{D4.1}\,\eqref{D4.1.2}, there exist a $\rho\in(0,\xi^2)$, $\bx_i\in\Omega_i\cap B_{\xi^2}(\bx)$ (hence, $\bx_i\in\Omega_i\cap B_{\eps}(\bx)$), and $A_i\in\Xi_i$ $(i=1,\ldots,n)$ such that
$\max_{1\le i\le n}d(\bx_i,A_i)<\xi^2\rho$ and
$\bigcap_{i=1}^n(A_i-x_i)\cap (\rho\B)=\emptyset$.
Hence, there exist points $\omega_i\in A_i$ $(i=1,\ldots,n)$ such that $\max_{1\le i\le n} \|\bx_i-\omega_i\|<\xi^2\rho$.
Set $\eps':=\xi^2\rho$ and $\de:=\xi^2\rho^{\frac12}$.
By Corollary~\ref{C3.20} with Remark~\ref{R3.1}\,\eqref{R3.1.4} in mind, there exist points $x_i\in A_i\cap B_{\de}(\omega_i)$,
$x_i'^*\in X^*$ $(i=1,\ldots,n)$ and $x\in\rho\B$ (hence, $x\in\eps\B$) such that $\sum_{i=1}^n \|x_i'^*\|=1$ and conditions \eqref{C3.1-3} and \eqref{C3.2-2} are satisfied with $A_i$, $x_i'^*$ $(i=1,\ldots,n)$ and $\eps'$ in place of $\Omega_i$, $x_i^*$ $(i=1,\ldots,n)$ and $\eps$.
Thus,
\begin{gather*}
\|x_i-\bx\|\le\|x_i-\omega_i\|+\|\omega_i-\bx_i\| +\|\bx_i-\bx\| <\de+\xi^2\rho+\xi^2<\xi^3+\xi^4+\xi^2<3\xi<\eps,
\\
\sum_{i=1}^n d\big(x_i'^*,N^C_{A_i}(x_i)\big) <\frac{\eps'}{\de}= \rho^{\frac12}<\xi,\quad \Big\|\sum_{i=1}^nx_i'^*\Big\|<\frac{\eps'}{\rho}=\xi^2,
\\
\sum_{i=1}^n\langle x_i'^*, x+\bx_i-x_i\rangle=\max_{1\le i\le n} \|x+\bx_i-x_i\|.
\end{gather*}
There exist $z_i^*\in N^C_{A_i}(x_i)$ $(i=1,\ldots,n)$ such that $\sum_{i=1}^n\|x_i'^*-z_i^*\|<\xi$.
Thus,
\begin{gather*}
0<1-\xi<\sum_{i=1}^n\|z_i^*\|<1+\xi\AND \Big\|\sum_{i=1}^nz_i^*\Big\|<\xi^2+\xi<2\xi.
\end{gather*}
Set $x_i^*:=z_i^*/\sum_{j=1}^n\|z_j^*\|$ $(i=1,\ldots,n)$.
Then $x_i^*\in N^C_{A_i}(x_i)$ $(i=1,\ldots,n)$,
\begin{gather*}
\sum_{i=1}^n\|x_i^*\|=1\AND \Big\|\sum_{i=1}^nx_i^*\Big\| <\frac{2\xi}{1-\xi}<4\xi<\eps.
\end{gather*}
%This proves \eqref{T3.2-2}.
Moreover,
\begin{align*}
\sum_{i=1}^n\|x_i^*-x_i'^*\|&\le\sum_{i=1}^n \Big\|\frac{z_i^*}{\sum_{j=1}^n\|z_j^*\|}-z_i^*\Big\| +\sum_{i=1}^n\|z_i^*-x_i'^*\|
\\&
=\Big|\sum_{j=1}^n\|z_j^*\|-1\Big| +\sum_{i=1}^n\|z_i^*-x_i'^*\|<\xi+\xi=2\xi,
\end{align*}
and consequently,
\begin{align}
\notag
\sum_{i=1}^n\langle x_i^*, x+\bx_i-x_i\rangle&>\sum_{i=1}^n\langle x_i'^*, x+\bx_i-x_i\rangle- 2\xi\max_{1\le i\le n} \|x+\bx_i-x_i\|
\\&=
\label{T4.5P10}
(1-2\xi)\max_{1\le i\le n}\|x+\bx_i-x_i\| >\tau\max_{1\le i\le n}\|x+\bx_i-x_i\|.
\end{align}
This proves \eqref{C3.2-3}.

Suppose $X$ is Asplund, and
let $\tau'\in(\tau+2\xi,1)$.
Application of the second part of Corollary~\ref{C3.20} with $\tau'$ in place of $\tau$ in the above proof justifies $\|\sum_{i=1}^nx_i^*\|<\eps$ and $\sum_{i=1}^n\|x_i^*\|=1$ with $x_i^*\in N^F_{A_i}(x_i)$ $(i=1,\ldots,n)$, while the factor $1-2\xi$ in \eqref{T4.5P10} needs to be replaced by $\tau'-2\xi$ leading to the same final estimate.
This again proves \eqref{C3.2-3}.
\end{proof}

\begin{remark}
%\label{R3.4+}
\begin{enumerate}
\item
The necessary conditions are formulated in Theorem~\ref{T4.5} for the approximate stationarity.
In view of Remark~\ref{R3}\,\eqref{R3.01}, if $\bx\in\cap_{i=1}^n\Omega_i$, they are also valid for the stationarity and extremality.
\item
Dropping the number $\tau$ together with the points $\bx_i$ $(i=1,2\ldots,n)$ and condition \eqref{C3.2-3} from the statement of Theorem~\ref{T4.5} produces an even simpler (though weaker) version of fuzzy separation; cf. Remark~\ref{R3.1}\,\eqref{R3.1.6}.

\item
Theorem~\ref{T4.5} shows that approximate stationarity of a given collection of families of closed sets implies its fuzzy (up to $\eps$) separation.
Note that, unlike the model discussed in \cite{BuiKru19}, not only the points $x_i$ and $x_i^*$ $(i=1,\ldots,n)$ usually involved in fuzzy separation statements depend on $\eps$, but also the perturbation sets $A_1,\ldots,A_n$.

\item
%\label{R3.4.3}
In the particular case when $\Xi_1,\ldots,\Xi_n$ are families of linear translations of the given sets
$\Omega_1,\ldots,\Omega_n$, i.e., $\Xi_i:=\{\Omega_i-a\mid a\in X\}$ (see Remark~\ref{R3}\,\eqref{R3.02}), the fuzzy separation in Theorem~\ref{T4.5} can be reformulated in terms of the sets
$\Omega_1,\ldots,\Omega_n$ instead of $A_1,\ldots,A_n$.
It suffices to observe that condition $x_i\in A_i$ in this case is equivalent to $x_i':=x_i+a_i\in\Omega_i$, while $N^C_{A_i}(x_i)=N^C_{\Omega_i}(x_i')$ and
$N^F_{A_i}(x_i)=N^F_{\Omega_i}(x_i')$.
Thus, Theorem~\ref{T4.5} covers the conventional extremal principle and its many extensions; cf. \cite{BuiKru18,BuiKru19}.
When the families of perturbations are defined by set-valued mappings (see Remark~\ref{R3}\,\eqref{R3.02}), Theorem~\ref{T4.5} enhances the approximate
extremal principle in \cite[Theorem~4.1]{MorTreZhu03}; see also \cite[Theorem~5.68\,(b)]{Mor06.2}.

\item
A more general quantitative version of the approximate stationarity property can be of interest.
Given an $\al>0$, one can replace inequality $\max_{1\le i\le n}d(x_i, A_i)<\varepsilon\rho$ in Definition~\ref{D4.1}\,\eqref{D4.1.2} by
$\max_{1\le i\le n}d(x_i, A_i)<\al\rho$ and talk about \emph{approximate $\al$-stationarity}; cf. \cite[Definition~3.1]{BuiKru19}.
The proof of Theorem~\ref{T4.5} can be easily adjusted to this setting producing a fuzzy separation result with inequality $\|\sum_{i=1}^nx_i^*\|<\eps$ replaced by $\|\sum_{i=1}^nx_i^*\|<\be$ where $\be$ is an arbitrary number greater than $\al$.

\item
Theorem~\ref{T4.5} can be `reversed' into a statement providing dual characterizations of the absence of the approximate stationarity, which can be interpreted as a kind of \emph{transversality} of collections of families of sets.
Such properties
%which
play an important role in constraint qualifications, qualification conditions in subdifferential/normal cone/coderivative calculus and convergence analysis of computational algorithms \cite{Kru05,Kru06,Kru09,KruLop12.1,KruTha13,KruTha15, KruLukTha18, LewLukMal09,DruIofLew15, Iof17}.
\end{enumerate}
\end{remark}

Theorem~\ref{T4.5} serves as a powerful tool for establishing optimality conditions for a broad class of set-valued optimization problems with general preference relations.
A simplified example of such a problem with a single geometric constraint can be written in the form:
\begin{gather}
\label{P}
\tag{$P$}
\text{minimize }\;F(x)\quad \text{subject to }\; x\in\Omega.
\end{gather}
Here $F:X\toto Y$ is a \SVM\ between normed spaces, $\Omega\subset X$, and the \emph{preference} is determined by a subset $K\subset Y$.

Optimality in problem \eqref{P} can be interpreted as (a kind of) extremality of the pair $\Omega_1:=\gph F$ and $\Omega_2:=\Omega\times K$ in $X\times Y$.
Given a nonempty family $\Xi$ of subsets of $Y$ (representing `perturbations' of $K$), one can define two families of subsets of $X\times Y$: $\Xi_1:=\{\Omega_1\}$ (a single set) and $\Xi_2:=\{\Omega\times\widetilde K\mid\widetilde K\in\Xi\}$ and apply the theory developed above.
The preference set $K$ can be given, e.g., as $K:=L(\by)\cup\{\by\}$, where $L:Y\toto Y$ is an abstract \emph{level-set mapping} (see \cite{KhaTamZal15}), and $\by$ is a fixed point in $Y$.
The details can be found in \cite{CuoKruTha}, where more general than \eqref{P} constrained set-valued optimization problems are studied.

\section*{Disclosure statement}

No potential conflict of interest was reported by the authors.

\section*{Funding}

Nguyen Duy Cuong is supported by Vietnam National Program for the Development of Mathematics 2021-2030 under grant number B2023-CTT-09.

\section*{Acknowledgments}

A part of the work was done during Alexander Kruger's stay at the Vietnam Institute for Advanced Study in Mathematics in Hanoi.
He is grateful to the Institute for its hospitality and supportive environment.

The authors thank the reviewer for his/her careful reading of the manuscript, detailed analysis of the approach and very specific comments and suggestions which led to significant changes in the manuscript improving its readability.

%\bibliography{BUCH-kr,Kruger,KR-tmp}
%\bibliographystyle{tfnlm}
\end{document}